\documentclass[12pt]{amsart}
\usepackage[dvips]{graphicx}
\usepackage{amsmath,graphics}
\usepackage{amsfonts,amssymb,hyperref}
\usepackage{xypic}
\usepackage{comment}
\specialcomment{proofc}{}{}
\includecomment{proofc}
\theoremstyle{plain}
\newtheorem*{theorem*}{Theorem}
\newtheorem*{lemma*} {Lemma}
\newtheorem*{corollary*} {Corollary}
\newtheorem*{proposition*}{Proposition}
\newtheorem*{conjecture*}{Conjecture}
\newtheorem{theorem}{Theorem}[section]
\newtheorem{lemma}[theorem]{Lemma}
\newtheorem*{theorem1*}{Theorem 1}
\newtheorem*{theorem2*}{Theorem 2}
\newtheorem*{theorem3*}{Theorem 3}
\newtheorem{corollary}[theorem]{Corollary}
\newtheorem{proposition}[theorem]{Proposition}

\newtheorem{question}[theorem]{Question}

\newcommand{\co}{\colon \thinspace}

\theoremstyle{remark}
\newtheorem*{remark}{Remark}

\newtheorem{example*}{Example}

\theoremstyle{definition}

\def\op{\operatorname}

\def\G{\Gamma}

\def\tor{\op{Tor}}

\def\gl{\op{GL}}   \def\Z{\Bbb{Z}} \def\R{\Bbb{R}} \def\C{\Bbb{C}}
\def\N{\Bbb{N}}    
 \def\a{\alpha} \def\g{\gamma} \def\tor{\op{Tor}} \def\bp{\begin{pmatrix}}
\def\sm{\setminus} \def\ep{\end{pmatrix}} \def\bn{\begin{enumerate}} 
   \def\en{\end{enumerate}}
\def\ba{\begin{array}} \def\ea{\end{array}} \def\L{\Lambda} 
 \def\S{\Sigma}  \def\a{\alpha} \def\b{\beta} \def\wti{\wtilde}

\def\ker{\op{ker}}\def\be{\begin{equation}} \def\ee{\end{equation}} 
  \def\ord{\op{ord}} 
 \def\hom{\op{Hom}}  
 \def\aut{\op{Aut}}  
 \def\dim{\op{dim}}  
\def\fkt{\F[t^{\pm 1}]^k} \def\ft{\F[t^{\pm
1}]}       
  \def\ztn{\Z[\Z_n]}
\def\fktn{\F[\Z_n]^k}
 \def\zpktn{\F_p[\Z_n]^k} 
\def\zt{\Z[t^{\pm 1}]}    
\def\ct{\C[t^{\pm 1}]}

\def\wti{\widetilde}

\def\what{\widehat}

\def\F{\Bbb{F}}

\begin{document}
\title[{The profinite completion, fiberedness and the Thurston norm}]{The profinite completion of $3$-manifold groups, fiberedness and the Thurston norm}
\author{Michel Boileau}
\address{
Institut de Math\'ematiques de Marseille\\Aix Marseille Universit\'e\\France}
\email{michel.boileau@cmi.univ-mrs.fr }

\author{Stefan Friedl}
\address{Fakult\"at f\"ur Mathematik\\ Universit\"at Regensburg\\   Germany}
\email{sfriedl@gmail.com}

\begin{abstract}
We show that a regular isomorphism of profinite completion of the fundamental groups of two 3-manifolds $N_1$ and $N_2$ induces an isometry of the Thurston norms and a bijection between the fibered classes. We study to what extent does the profinite completion of knot groups distinguish knots and show that it distinguishes each torus knot and the figure eight knot among all knots. We show also that it distinguishes between hyperbolic knots with cyclically commensurable complements under the assumption that their Alexander polynomials have at least one zero which is not a root of unity.
\end{abstract}

\maketitle

\hfill {\emph{Dedicated to the memory of William Thurston}
\section{Introduction}

In this paper we will study the question, what properties of 3-manifolds are determined by the set of finite quotients of their fundamental groups.
The modern reformulation of the above question (see Lemma~\ref{lem:finitequotients} for details) is to ask, what properties of a  3-manifold $N$ are determined by the profinite completion $\what{\pi_1(N)}$ of its fundamental group. Here, and throughout the paper all 3-manifolds are understood to be compact, orientable, connected with empty or toroidal boundary.

Wilton--Zalesskii \cite{WZ14} showed that the profinite completion of the fundamental group of a closed $3$-manifold  can detect whether or not it is hyperbolic, furthermore they showed that it can detect whether or not it is Seifert fibered. On the other hand it is known by work of Funar \cite{Fu13} and Hempel \cite{He14} that the profinite completion of the fundamental group can not always distinguish between pairs of torus bundles and between certain pairs of Seifert manifolds. It is still an open question though whether the profinite completion can distinguish any two hyperbolic 3-manifolds.

In this paper we are mostly interested in the relation between the profinite completion,  the fiberedness and the Thurston norm of a 3-manifold.
We quickly recall the relevant definitions. 
Given a surface $\S$ with connected components $\S_1,\dots,\S_k$  its complexity
is defined to be
\[ \chi_-(\S):=\sum_{i=1}^d \max\{-\chi(\S_i),0\}.\]
Given a 3-manifold $N$ and $\phi \in H^1(N;\Z)$ the Thurston norm is defined as
\[ x_N(\phi):= \min\{\chi_-(\S)\, |\, \S \subset N\mbox{ properly embedded and dual to }\phi\}.\]
 Thurston
\cite{Th86} showed that  $x_N$ is a seminorm on $H^1(N;\Z)$. It follows from standard arguments that $x_N$ extends to a seminorm on $H^1(N;\R)$.
We also recall that  an integral class $\phi \in H^1(N;\Z)=\hom(\pi_1(N),\Z)$ is called \emph{fibered}  if  there exists a fibration $p\colon N\to S^1$ such that
$\phi=p_*\colon \pi_1(N)\to \Z$. More generally, we say that a class $\phi \in H^1(N;\R)$ is  \emph{fibered}  if $\phi$ can be represented by a nowhere-vanishing closed 1--form. By \cite{Ti70}  the two notions of being fibered coincide for integral cohomology classes.

In the following let $N_1$ and $N_2$ be two 3-manifolds and suppose there exists an isomorphism $f\colon \what{\pi_1(N_1)}\to \what{\pi_1(N_2)}$ of the profinite completions of the fundamental groups. Such an isomorphism induces an 
isomorphism
$ \what{H_1(N_1;\Z)}\to \what{H_1(N_2;\Z)}$ of the profinite completions of the homology groups. It is straightforward to see that this implies that 
$H_1(N_1;\Z)$ and $H_1(N_2;\Z)$ are abstractly isomorphic. But in general the isomorphism 
$ \what{H_1(N_1;\Z)}\to \what{H_1(N_2;\Z)}$  is not  induced  by an isomorphism of the homology groups.
 Since we want to compare the Thurston norm and the fibered classes of $N_1$ and $N_2$ it is convenient to assume that $ \what{H_1(N_1;\Z)}\to \what{H_1(N_2;\Z)}$ is in fact induced by an isomorphism   $ H_1(N_1;\Z)\to H_1(N_2;\Z)$.

This leads us to the following definition: we say that an isomorphism $f\colon \what{\pi_1}\to \what{\pi_2}$ between profinite completions of two groups $\pi_1$ and $\pi_2$ is 
\emph{regular}  
if the induced isomorphism $ \what{H_1(\pi_1;\Z)}\to \what{H_1(\pi_2;\Z)}$ is induced by an isomorphism $H_1(\pi_1;\Z)\to H_1(\pi_2;\Z)$. This isomorphism is then necessarily uniquely determined by $f$ and by a slight abuse of notation we also denote it by $f$.

Throughout this paper we will restrict ourselves to 3-manifolds that are aspherical and with empty or toroidal boundary. It follows from the Geometrization Theorem that a 3-manifold, with the adjectives mentioned above, is aspherical if and only if it is prime, has infinite fundamental group and is not homeomorphic to $S^1\times S^2$. 

Now we can formulate our main theorem.

\begin{theorem}\label{mainthm}
Let $N_1$ and $N_2$ be two aspherical 3-manifolds with empty or toroidal boundary. Suppose $f\colon \what{\pi_1(N_1)}\to \what{\pi_1(N_2)}$ is a regular isomorphism.  Let $\phi\in H^1(N_2;\R)$.
Then 
\[ \phi\in H^1(N_2;\R)\mbox{ is fibered}\quad \Longleftrightarrow \quad f^*\phi\in H^1(N_1;\R)\mbox{ is fibered}.\]
Furthermore
\[ x_{N_2}(\phi)=x_{N_1}(f^*\phi).\]
\end{theorem}

The proof rests on the fact that 3-manifold groups are good in the sense of Serre \cite{Se97} and that profinite completions contain enough informations on certain twisted Alexander polynomials which by \cite{FV12,FV13,FN15} determine fiberedness and the Thurston norm. Both types of results rely on the recent work of Agol \cite{Ag08,Ag13}, Przytycki--Wise \cite{PW12} and Wise \cite{Wi09,Wi12a,Wi12b}.

For the remainder of this introduction we consider the special case of knot complements. Given a knot $K\subset S^3$ we denote by $X(K)=S^3\sm \nu K$ its exterior and we denote by $\pi(K):=\pi_1(S^3\sm \nu K)$ the corresponding knot group. We say that \emph{$K$ is fibered} if the knot exterior $X(K)$ is a surface bundle over $S^1$. Furthermore, we refer to the minimal genus of a Seifert surface for $K$ as the \emph{genus $g(K)$ of $K$}.

Our main theorem for knots is the following variation on Theorem~\ref{mainthm}.

\begin{theorem}\label{mainthmknots}
If $J$ and $K$ are two knots such that the profinite completions of their  groups are isomorphic, then $J$ is fibered if and only if $K$ is fibered.
Furthermore, $g(J)=g(K)$.
\end{theorem}

This theorem is not an immediate corollary to Theorem~\ref{mainthm} since we do not assume that the isomorphism of the profinite completions is regular.

It is also natural to ask, to what degree does the profinite completion of knot groups distinguish knots?
Here, and throughout the paper we say that two knots $J$ and $K$ are \emph{equivalent} if there exists a diffeomorphism $h$ of $S^3$ with $h(J)=K$.
Evidently the profinite completion of knot groups can only determine knots which are already determined by their groups.
Whitten \cite{Wh87} showed that prime knots are determined by their groups. 
The following question thus arises.

\begin{question}\label{question:primeknots}
Let $J$ and $K$ be two prime knots with $\what{\pi(J)}\cong \what{\pi(K)}$. Does it follow that $J$ and $K$ are equivalent?
\end{question}

It is straightforward to see that the profinite completion detects the unknot. For completeness' sake we provide the proof in Lemma~\ref{lem:profinitetrivial}.  Also, in many real-life situations it can be quite easy to distinguish two given knots by their finite quotients.

Using Theorem~\ref{mainthmknots} we quickly obtain the following corollary.

\begin{corollary}\label{cor:distinguish34-intro} 
Let $J$ be the trefoil or the Figure-8 knot. If $K$ is a knot with $\what{\pi(J)}\cong \what{\pi(K)}$, then $J$ and $K$ are equivalent.
\end{corollary}

More generally we prove the following result for torus knots.

\begin{theorem}\label{thm:torusknots}
Let $J$ be a torus knot. If $K$ is a knot with $\what{\pi(J)}\cong \what{\pi(K)}$, then $J$ and $K$ are equivalent.
\end{theorem}

Two knots in the $3$-sphere are said commensurable if their exteriors have homeomorphic finite-sheeted covers. If these homeomorphic finite-sheeted covers are cyclic covers, then the knots  
are said cyclically commensurable. The following result shows that cyclically commensurable hyperbolic knots are distinguished by their finite quotients when their Alexander polynomial are not a product of cyclotomic polynomials.

\begin{theorem}\label{thm:commensurable} Let $K_1$ and $K_2$ be two cyclically commensurable hyperbolic knots. If the Alexander polynomial of $K_1$ has at least one zero which is not a root of unity and $\what{\pi(K_1)}\cong \what{\pi(K_2)}$, then $K_1$ and  $K_2$ are equivalent.
\end{theorem}

A hyperbolic knot $K$ has no hidden symmetries if the commensurator of $\pi(K)$ in $PSL(2, \mathbb{C})$ coincides with its normalizer. The following corollary is a straightforward consequence of Theorem \ref{thm:commensurable} above  and the fact that two hyperbolic knots in $S^3$ without hidden symmetries are commensurable iff they are cyclically commensurable, see \cite[Thm 1.4]{BBCW12}.

\begin{corollary} Let $K_1$ and $K_2$ be two commensurable hyperbolic knots in $S^3$ without hidden symmetries. If the Alexander polynomial of $K_1$ has at least one zero which is not a root of unity and $\what{\pi(K_1)}\cong \what{\pi(K_2)}$, then $K_1$ and $K_2$ are equivalent.
\end{corollary}

\begin{remark}

\noindent {\bf 1.} Theorem \ref{thm:commensurable} applies in particular to the  the pretzel knot $K$ of type $(−2, 3, 7)$. It is known as the Fintushel-Stern knot $K$and  admits two lens space surgeries. So its exterior  is cyclically covered by the exteriors of two distinct knots. It is a fibered knot with Alexander polynomial the Lehmer polynomial $\Delta_K =  1 + t - t^3 - t^4 - t^5 - t^6 - t^7 + t^9  + t^{10}$ which is the integral polynomial of smallest known Mahler measure (see Section \ref{section:commensurable} for the definition) and which has two real roots, one being the Salem number $1.17628....$ .

\noindent {\bf 2.} Currently, the only hyperbolic knots known to admit hidden symmetries are the figure-8 and the two dodecahedral knots of Aitchison and Rubinstein, see \cite{AR92}.
By \cite{Re91} the figure-8 is the only  knot in $S^3$  with arithmetic complement, hence it is the unique knot in its commensurability class. 
The two dodecahedral knots are commensurable, but one is fibered and the other one not, so it follows from Theorem \ref{mainthmknots} that their groups cannot have the same profinite completion.
\end{remark}

The paper is structured as follows.
In Section~\ref{section:prelim} we recall the definition and some basic facts on  profinite completions. In Section~\ref{section:tap} we recall the definition of twisted homology and cohomology groups and of twisted Alexander polynomials of $3$-manifolds. In Section~\ref{section:mainthm} we will relate profinite completions to the degrees of certain twisted Alexander polynomials.
This will then allow us to prove Theorem~\ref{mainthm} and a slightly stronger version of Theorem~\ref{mainthmknots}. In Section~\ref{section:torusknots} we prove Theorem \ref{thm:torusknots}   about torus knots and in Section \ref{section:commensurable} we prove Theorem \ref{thm:commensurable} about cyclically commensurable hyperbolic knots.

There are some overlaps between the methods and the results of this article and of a forthcoming article by Martin Bridson and Alan Reid.

\subsection*{Convention.} Unless it says specifically otherwise, all groups are assumed to be finitely generated, all modules are assumed to be finitely generated, all manifolds are assumed to be orientable, connected and compact, and all 3-manifolds are assumed to have empty or toroidal boundary. 
Finally, $p$ will always be a prime and $\F_p$ denotes the field with $p$ coefficients.

\subsection*{Acknowledgment.} We wish to thank  Baskar Balasubramanyam, Marc Lackenby, Jacob Rasmussen, Alan Reid, Henry Wilton and Pavel Zalesskii for helpful conversations. The second author is grateful for the hospitality at  Keble College, IISER Pune, Universit\'e de Toulouse and Aix-Marseille Universit\'e. Finally the second author was supported by the SFB 1085 `Higher Invariants' at the University of Regensburg, funded by the Deutsche Forschungsgemeinschaft (DFG).

\section{The profinite completion of a group}\label{section:prelim}
In this section we recall several basic properties of profinite completions. Throughout this section we refer to \cite{RZ10} and \cite{Wi98} for details.

\subsection{The definition of the profinite completion of a group}
Given a group $\pi$  we consider the inverse system
$\{\pi/\G\}_{\G}$ where $\G$ runs over all finite index normal subgroups of $\pi$. 
The  profinite completion $\what{\pi}$ of $\pi$ is then defined as the inverse limit
of this system, i.e.\
\[ \what{\pi}=\underset{\longleftarrow}{\lim} \,\pi/\G. \]
Note that the natural map $\pi\to \what{\pi}$ is injective if and only if  $\pi$ is residually finite. It follows from  \cite{He87} and the proof of the Geometrization Conjecture that  fundamental groups of 3-manifolds are residually finite.

When $\pi$ is finitely generated, a deep result in \cite{NS07} states that every finite index subgroup of $\what{\pi}$ is open. It means that $\what{\what{\pi}} = \what{\pi}$. Then the following is a consequence of the definitions, see \cite[Proposition~3.2.2]{RZ10} for details.

\begin{lemma}\label{lem:samefinitequotients}
Let $\pi$ be a finitely generated group. Then for any finite group $G$ the map $\pi\to \what{\pi}$ induces a bijection
$\hom(\what{\pi},G)\to \hom(\pi,G)$.
\end{lemma}

\subsection{Groups with isomorphic profinite completions}


In this section $A$ and $B$ are finitely generated groups.
A group homomorphism $\varphi\colon A\to B$ induces a homomorphism $\what{\varphi}\colon \what{A}\to\what{B}$.
Evidently, if $\varphi$ is an isomorphism, then so is $\what{\varphi}$. 
On the other hand,  an isomorphism $\phi \co \what{A}\to \what{B}$  is not necessarily induced by a homomorphism $\varphi\co A\to B$. There are even isomorphisms $\what{\Z}\to \what{\Z}$ that are not induced by an automorphism of $\Z$. 
Moreover, it follows from \cite{NS07} that any homomorphism $\phi \colon \what{A}\to\what{B}$ is in fact continuous.

If $f\colon \what{A}\to\what{B}$ is an isomorphism, then it follows from Lemma~\ref{lem:samefinitequotients}
that for any finite group $G$ we have bijections
\[  \hom(B,G)\leftarrow   \hom(\what{B},G) \xrightarrow{f^*} \hom(\what{A},G)\to \hom(A,G).\]
Given $\b\in \hom(B,G)$ we will, by a slight abuse of notation, denote by $\b\circ f$ the resulting homomorphism from $A$ to $G$.
In particular  given a representation $\b\colon B\to \gl(k,\F_p)$ we obtain an induced representation 
$A\to \gl(k,\F_p)$ which we  denote by $\b\circ f$. 
\medskip

Given a group $\pi$ we denote by $Q(\pi)$ the set of finite quotients of $\pi$.
We just showed that  finitely generated groups $A$ and $B$ with isomorphic profinite completions have the same finite quotients,
i.e.\ the sets $Q(A)$ and $Q(B)$ are the same.
Somewhat surprisingly the converse also holds, more precisely, by \cite[Corollary~3.2.8]{RZ10} the following lemma holds.

\begin{lemma}\label{lem:finitequotients}
Two finitely generated groups $A$ and $B$ have isomorphic profinite completions if and only if $Q(A)=Q(B)$.
\end{lemma}


\section{Twisted Alexander polynomials of knots} \label{section:tap}

\subsection{Definition of twisted homology and cohomology groups}

Let $X$ be a connected CW-complex. We write $\pi=\pi_1(X)$
and we denote by $p\colon \wti{X}\to X$ the universal covering of $X$.  Note that $\pi$ acts on $\wti{X}$ on the left via deck transformations. We can thus view $C_*(\wti{X})$ as a left $\Z[\pi]$-module.
Let $R$ be a commutative ring and let $V$ be an $R$-module. Let $\a\colon \pi\to \aut_R(V)$ be a representation. We henceforth view $V$ as a left $\Z[\pi]$-module.
Given any $i$ we  refer to 
\[ H^i_\a(X;V):=H_i\big(\hom_{\Z[\pi]}\big(C_*(\wti{X}),V\big)\big)\]
as the \emph{$i$-th twisted cohomology of $(X,\a)$}. 

 Using the standard involution $g\mapsto g^{-1}$ we can turn the $\Z[\pi]$-left module  $C_*(\wti{X})$ into a right $\Z[\pi]$-module.
 Given any $i$ we then refer to 
 \[ H_i^\a(X;V):=H_i\big(C_*(\wti{X})\otimes_{\Z[\pi]}V\big)\]
 as the \emph{$i$-th twisted homology of $(X,\a)$}.

\subsection{Orders of modules}
Let $\F$ be a field and let  $H$ be a finitely generated $\ft$-module. Since $\ft$ is a PID there exists an isomorphism $H\cong \oplus_{i=1}^n \ft/f_i(t)$ where $f_1(t),\dots,f_n(t)\in \ft$. 
We refer to $\ord(H):=\prod_{i=1}^n f_i(t)\in \ft$ as the \emph{order of $H$}. 
 The order is well-defined up to multiplication by a unit in $\ft$. 
 Furthermore it is non-zero if and only if  $H$ is an $\ft$-torsion module.
 
\subsection{The definition of twisted Alexander polynomials}
\label{section:deftap}
%

Let $X$ be a CW-complex, let $\phi\in H^1(X;\Z)$ and let  $\a\colon \pi_1(X)\to \gl(k,\F)$ be a representation over a field $\F$. We write $\pi=\pi_1(X)$ and $\fkt:=\F^k\otimes_{\Z} \zt$ and we denote by $\a\otimes \phi$ the tensor representation
\[ \ba{rcl} \a\otimes \phi\colon \pi&\to &\aut_{\ft}(\fkt)\\[2mm]
g&\mapsto & \left(\ba{rcl} \fkt&\to &\fkt\\ \sum_i v_i\otimes p_i(t)&\mapsto& \sum_i\a(g)(v_i)\otimes t^{\phi(g)}p_i(t)\ea\right).\ea \]
This allows us to view $\fkt$ as a left $\Z[\pi]$-module. 
We then consider the twisted homology groups $H_i^{\a\otimes \phi}(X;\fkt)$ which are naturally $\ft$-modules. Given $i\in \N$ we denote by $\Delta_{X,\phi,i}^\a\in \ft$ the order of the $\ft$-module $H_i^{\a\otimes \phi}(X;\fkt)$
and we refer to it as the  \emph{$i$-th twisted Alexander polynomial of $(X,\phi,\a)$}. We refer to the original papers \cite{Li01,Wa94,Ki96,KL99} and the survey papers \cite{FV10,DFL14} for more information on twisted Alexander polynomials.

The twisted Alexander polynomials are well-defined up to multiplication by a unit in $\ft$, i.e.\ up to multiplication by some $at^k$ where $a\in \F\sm\{0\}$ and $k\in \Z$. In the following, given $p,q\in \ft$ we write $p\doteq q$ if $p$ and $q$ agree up to multiplication by a unit in $\ft$.

For future reference we recall two lemmas about twisted Alexander polynomials. 
The first lemma  is \cite[Lemma~2.4]{FK06}.

\begin{lemma}\label{lem:delta0}
Let $X$ be a CW-complex, let $\phi\in H^1(X;\Z)$ be non-zero and let  $\a\colon \pi_1(X)\to \gl(k,\F)$ be a representation over a field $\F$.
Then the zeroth twisted Alexander polynomial $\Delta_{X,\phi,0}^\a$ is non-zero.
\end{lemma}

Given a representation $\a\colon \pi\to \gl(k,\F_p)$ we denote by $\a^*\colon\pi\to \gl(k,\F_p)$ the representation which is given by 
$\a^*(g):=\a(g^{-1})^t$ for $g\in \pi$.
The following lemma is \cite[Proposition~2.5]{FK06}.

\begin{lemma}\label{lem:delta2}
Let $N$ be a 3-manifold, let $\phi\in H^1(N;\Z)$ be non-zero and let  $\a\colon \pi_1(N)\to \gl(k,\F)$ be a representation over a field $\F$. Suppose that 
$\Delta_{N,\phi,1}^\a$ is non-zero. 
Then the following hold:
\bn
\item If $N$ has non-trivial boundary, then  $\Delta_{N,\phi,2}^\a\doteq 1$.
\item If $N$ is closed, then $ \Delta_{N,\phi,2}^\a\doteq \Delta_{N,\phi,0}^{\a^*}$.
\en
\end{lemma}

\subsection{Twisted Alexander polynomials, fiberedness and the Thurston norm}
Given a polynomial $f(t)=\sum_{k=r}^s a_kt^k\in \ft$ with $a_r\ne 0$ and $a_s\ne 0$ we define $\deg(f(t))=s-r$. We extend this definition to the zero polynomial by setting $\deg(0):=+\infty$. It is clear that the degree of a twisted Alexander polynomial is well-defined, i.e.\ not affected by the indeterminacy in the definition.

\begin{theorem}\label{thm:fv}
Let $N$ be a 3-manifold and let $\phi\in H^1(N;\Z)$ be non-zero. Then the following hold:
\bn
\item Pick a prime $p$. The class $\phi$ is fibered if and only if for any representation $\a\colon \pi_1(N)\to \gl(k,\F_p)$ we have $\Delta_{N,\phi,1}^\a \ne 0$.
\item Pick a prime $p$. Let $\a\colon \pi_1(N)\to \gl(k,\F_p)$ be a representation such that   $\Delta_{N,\phi,1}^\a$ is non-zero. Then the  twisted Alexander polynomials $\Delta_{N,\phi,i}^\a \ne 0$ are non-zero for $i=0,1,2$ and 
\[ x_N(\phi)\geq \max\left\{0,\frac{1}{k}\left(-\deg \left(\Delta_{N,\phi,0}^\a\right)+\deg \left(\Delta_{N,\phi,1}^\a\right)-\deg \left(\Delta_{N,\phi,2}^\a\right)\right)\right\}.\]
\item If $N$ is aspherical, then there exists a prime  $p$ and a representation $\a\colon \pi_1(N)\to \gl(k,\F_p)$ such that  
\[ x_N(\phi)=\max\left\{0, \frac{1}{k}\left(-\deg \left(\Delta_{N,\phi,0}^\a\right)+\deg \left(\Delta_{N,\phi,1}^\a\right)-\deg \left(\Delta_{N,\phi,2}^\a\right)\right)\right\}.\]
\en
\end{theorem}

Here the `only if' direction of (1) was proven by various authors, see e.g.\ \cite{Ch03,GKM05,FK06}. The `if' direction was proven in \cite{FV13}. 
The inequality in (2) was proved in \cite{FK06,Fr14}. Finally statement (3) is proven in \cite{FN15} building on \cite{FV12}. Here the `if' statement of (1) and the proof of (3) build on the work of Agol \cite{Ag08,Ag13}, Przytycki--Wise \cite{PW12} and Wise \cite{Wi09,Wi12a,Wi12b}.

\subsection{Degrees of twisted Alexander polynomials}
Throughout the paper, given a group $\pi$, $\phi\in H^1(\pi;\Z)=\hom(\pi,\Z)$ and $n\in \N$ we denote by $\phi_n\colon \pi\xrightarrow{\phi} \Z\to \Z_n$ the composition of $\phi$ with the obvious projection map. 
Furthermore, given a representation $\a\colon \pi\to \gl(k,\F)$ and $n\in \N$ we write $\fktn=\F^k\otimes_{\Z} \ztn$ and we denote by $\a\otimes \phi_n\colon \pi\to \aut(\fktn)$ the  representation
which is defined in a completely analogous way as we defined $\a\otimes \phi$ above. Later on we will make use of the following proposition.

\begin{proposition}\label{prop:tapdegree}
Let $X$ be a CW-complex, let $\phi\in H^1(X;\Z)$ be non-trivial and let $\a\colon \pi_1(X)\to \gl(k,\F_p)$ be a representation. Then the following equalities hold:
\[ \ba{rcl} \deg \Delta_{X,\phi,0}^\a(t)&\hspace{-0.2cm}=\hspace{-0.2cm}& \max\left\{\hspace{-0.1cm} \dim_{\F_p}\hspace{-0.1cm}\left( H_0^{\a\otimes \phi_n}(X;\zpktn)\right)\bigg| n\in \N\right\}\\[2mm]
\deg \Delta_{X,\phi,1}^\a(t)&\hspace{-0.2cm}=\hspace{-0.2cm}& \max\left\{ \hspace{-0.1cm}\dim_{\F_p}\hspace{-0.1cm}\left( H_1^{\a\otimes \phi_n}(X;\zpktn)\right)- \dim_{\F_p}\hspace{-0.1cm}\left( H_0^{\a\otimes \phi_n}(X;\zpktn)\right)\bigg| n\in \N\right\}.\ea\]
\end{proposition}

The proof of this proposition will require the remainder of this section.
For simplicity we will henceforth write $\F=\F_p$, $\Lambda=\ft$ and for each $n\in \N$ we write $\Lambda_n=\ft/(t^n-1)$.
We first recall several elementary lemmas.

\begin{lemma}\label{lem:dimdeg}
For any $\L$-module $M$ we have  $ \dim_\F(M)=\deg(\ord(M))$.
\end{lemma}

\begin{proof}
We first observe that 
for any  polynomial $q(t)\in \L$, not necessarily non-zero, we  have $\dim_{\F}(\L/q(t)\L)=\deg(q(t))$. Since $\L$ is a PID the general case of the lemma follows  immediately from the cyclic case.
\end{proof}

\begin{lemma}\label{lem:tensordim}
Let $M$ be an $\L$-module. Then
\[ \dim_{\F}(M)=\max\left\{ \dim_{\F}\left(M\otimes_{\L}\L_n\right)\,\,\Big|\,\,n\in \N\right\}.\]
\end{lemma}

\begin{proof}
We write $M=\L^r\oplus T$ where $T$ is a torsion $\L$-module.
First we consider the case that $r>0$. In this case  $\dim_{\F}(M)=\infty$. On the other hand, for any $n$ we have 
\[ \dim_{\F}\left(M\otimes_{\L}\L_n\right)\geq \dim_{\F}\left(\L^r\otimes_{\L}\L_n\right)=\dim_{\F}\left(\L_n^r\right)=rn.\]
Thus we  showed that the claimed equality holds if $r>0$.

Now suppose that $r=0$. For any $n$ the ring epimorphism $\L\to \L_n$ induces an epimorphism $M=M\otimes_{\L}\L\to M\otimes_{\L}\L_n$. We thus see that for any $n$ we have 
\[ \dim_{\F}(M)\geq  \dim_{\F}\left(M\otimes_{\L}\L_n\right).\]
Since $M$ is in particular a finite abelian group there exists an $n$ such that multiplication by $t^n$ acts like the identity on $M$. Put differently, multiplication by $t^n-1$ is the zero map. For such $n$ it is straightforward to see that the map
\[ \ba{rcl} M\otimes_{\L}\L_n&\to & M\\
m\otimes [q(t)]&\to&mq(t)\ea\]
is a well-defined isomorphism of $\F$-modules. In particular 
$\dim_{\F}( M)=\dim_{\F}(M\otimes_{\L}\L_n)$. Together with the above inequality this implies the lemma.
 \end{proof}

\begin{lemma}\label{lem:utc}
Let $C_*$ be a chain complex of modules over $\L$ such that $H_0(C_*)$ is $\L$-torsion. Then 
\[ \ba{rcl}H_0(C_*\otimes_{\L}\L_n)&=& H_0(C_*)\otimes_{\L}\L_n\\[2mm]
 H_1(C_*\otimes_{\L}\L_n)&\cong& H_1(C_*)\otimes_{\L}\L_n\,\,\oplus\,\, H_0(C_*)\otimes_{\L}\L_n.\ea \]
\end{lemma}

\begin{proof}
By the universal coefficient theorem for chain complexes of modules over the PID $\L$ we have for any $i$ that 
\[ H_i(C_*\otimes_{\L}\L_n)\cong H_i(C_*)\otimes_{\L}\L_n\oplus \tor_{\L}\left(H_{i-1}(C_*),\L_n\right).\]
The lemma now follows easily from the definitions and from the fact that for any $\L$-module $H$ we have 
\[ \tor_{\L}\left(H,\L_n\right)\cong \tor_{\L}(H)\otimes_{\L}\L_n.\]
\end{proof} 

Now we are finally in a position to prove Proposition~\ref{prop:tapdegree}.

\begin{proof}[Proof of Proposition~\ref{prop:tapdegree}]
Let $X$ be a CW-complex, let $\phi\in H^1(X;\Z)$ be non-trivial and let $\a\colon \pi_1(X)\to \gl(k,\F_p)$ be a representation. As usual we denote by $\wti{X}$ the universal cover of $X$. We consider the $\L$-chain complex
\[ C_*:=C_*(\wti{X})\otimes_{\Z[\pi_1(X)]}\L^k.\]
With this notation we have
\[ H_i(X;\L^k)=H_i(C_*)\mbox{ and }H_i(X;\L_n^k)=H_i(C_*\otimes_{\L}\L_n).\]
(Here and throughout the proof  we drop the representation in the notation for twisted homology groups.)  By Lemma~\ref{lem:delta0} we know that $H_0(X;\Lambda^k)$ is $\Lambda$-torsion. Thus in the following we can apply Lemma~\ref{lem:utc} to the chain complex $C_*$. 

First we  consider $ \deg \Delta_{X,\phi,0}^\a(t)$.
It follows from Lemmas~\ref{lem:dimdeg}, \ref{lem:tensordim} and~\ref{lem:utc} that 
\[ \ba{rcl} \deg \Delta_{X,\phi,0}^\a(t)&=& \dim_{\F}\left(H_0(X;\Lambda^k)\right)\\[2mm]
&=& \max\left\{ \dim_{\F}\left(H_0(X;\Lambda^k)\otimes_{\Lambda}\Lambda_n\right)\,\,\big|\,\, n\in \N\right\}\\[2mm]
&=& \max\left\{ \dim_{\F}\left( H_0(X;\Lambda^k_n)\right)\,\,\big|\,\, n\in \N\right\}\ea\]

Now we turn to the proof of the second equality.
It  follows from applying Lemma~\ref{lem:dimdeg} and ~\ref{lem:tensordim} once and from applying Lemma~\ref{lem:utc} twice that 
\[ \ba{rcl} \deg \Delta_{X,\phi,1}^\a(t)&=& \dim_{\F}\left(H_1(X;\Lambda^k)\right)\\[2mm]
&=& \max\left\{ \dim_{\F}\left(H_1(X;\Lambda^k)\otimes_{\Lambda}\Lambda_n\right)\,\,\big|\,\, n\in \N\right\}\\[2mm]
&=& \max\left\{ \dim_{\F}\left(H_1(X;\Lambda_n^k)\right)-\dim_{\F} \left(H_0(X;\L^k)\otimes_{\L}\L_n\right)\,\,\big|\,\, n\in \N\right\}\\[2mm]
&=& \max\left\{ \dim_{\F}\left(H_1(X;\Lambda_n^k)\right)-\dim_{\F} \left(H_0(X;\L_n^k)\right)\,\,\big|\,\, n\in \N\right\}\ea\]
\end{proof}

\section{The profinite completion and twisted Alexander polynomials}
\label{section:mainthm}

\subsection{Twisted homology groups and profinite completions}
Following Serre \cite[D.2.6~Exercise~2]{Se97} 
 we say that a group $\pi$ is  \emph{good} if the following holds:
 for any representation $\a\colon \pi\to \aut_{\Z}(A)$, where $A$ is a finite abelian group,  the inclusion $\iota\co \pi\to \what{\pi}$ induces
 for any $i$ an isomorphism  $\iota^*\colon H^i_\a(\widehat{\pi};A)\to H^i_\a(\pi;A)$.
 
The following theorem was first proved by Cavendish  \cite[Section~3.5,~Lemma~3.7.1]{Ca12}. We also refer to  \cite[(H.26)]{AFW15} for an alternative approach which builds on \cite{WZ10} and the work of  Agol \cite{Ag08}, Przytycki--Wise \cite{PW12} and  Wise \cite{Wi09,Wi12a,Wi12b}.

\begin{theorem}\label{thm:cav}
All $3$-manifold groups are good. 
\end{theorem}

The fact that 3-manifolds are good gives us the following useful corollary.

\begin{corollary}\label{cor:profiniteclosedn}
Let $N_1$,~$N_2$ be   aspherical $3$-manifolds such that 
$\widehat{\pi_1(N_1)}\cong \widehat{\pi_1(N_2)}$. Then  $N_1$ is closed if and only if  $N_2$ is closed.
\end{corollary}

\begin{proof}
By goodness and by assumption we have isomorphisms
\[  H^3(\pi_1(N_1);\Z_2)\xleftarrow{\cong }
H^3(\widehat{\pi_1(N_1)};\Z_2)\cong 
H^3(\widehat{\pi_1(N_2)};\Z_2)\xrightarrow{\cong} H^3(\pi_1(N_2);\Z_2).\]
The corollary follows from the fact that an aspherical 3-manifold is closed if and only if the third cohomology group with $\Z_2$-coefficients is non-zero.
\end{proof}

The main goal of this section is to prove the following proposition.

\begin{proposition}\label{prop:sametwistedhomologies}
Let $\pi_1$ and $\pi_2$ be good groups. Let  $f\colon \what{\pi_1}\xrightarrow{\cong} \what{\pi_2}$ be an isomorphism. Let $\b\colon \pi_2\to \gl(k,\F_p)$ be a representation. Then for any $i$ we have an isomorphism
\[ H_i^{\b\circ f}(\pi_1;\F_p^k)\cong H_i^\b(\pi_2;\F_p^k).\]
\end{proposition}

Before we give a proof of the proposition we need to formulate a lemma which relates twisted cohomology groups to twisted homology groups. Here we recall that 
given a representation $\a\colon \pi\to \gl(k,\F)$ over a field $\F$ we denote by $\a^*\colon\pi\to \gl(k,\F)$ the representation which is given by 
$\a^*(g):=\a(g^{-1})^t$ for $g\in \pi$.

\begin{lemma}\label{lem:homcohom}
Let $\pi$ be a group and let  $\g\colon\pi\to \gl(k,\F)$ be a representation over a field $\F$.
Then $H^i_{\g^*}(\pi;\F^k)\cong H_i^\g(\pi;\F^k)$ for any $i$.
\end{lemma}

\begin{proof}
Let $Y$ be a $K(\pi,1)$. 
As usual we denote by $\wti{Y}$ the universal cover of $Y$. 
In the following we denote by $\F^k_{\g}$, respectively $\F^k_{\g^*}$, the vector space $\F^k$ together with the structure as a left $\Z[\pi]$-module induced by the representation $\g$, respectively $\g^*$.
A direct calculation shows that 
\[ \ba{rcl} \hom_{\Z[\pi]}\left(C_*(\wti{Y}),\F^k_{\g^*}\right)&\to& \hom_\F\left(C_*(\wti{Y})\otimes_{\Z[\pi]}\F^k_{\g},\F\right)\\
(\varphi\colon C_i(\wti{Y})\to \F^k)&\mapsto&\left(\ba{rcl} C_i(\wti{Y})\otimes_{\Z[\pi]}\F^k&\to & \F \\ \sigma\otimes v&\mapsto& \varphi(\sigma)^t\,p(v)\ea\right)\ea\]
is an isomorphism of chain complexes of $\F$-vector spaces.
 It follows from this isomorphism and from the  Universal Coefficient Theorem
applied to the chain complex $C_*(\wti{Y})\otimes_{\F[\pi]}\F_{\g}^k$ and to the field $\F$  that for any $i$ we have an isomorphism
\[
  H^i_{\g^*}(Y;\F^k)\cong 
\hom_{\F}\left(H_{i}^\g(Y;\F^k),\F\right).\]
Now the claim follows from the fact that a finite-dimensional vector space is isomorphic to its dual vector space.
\end{proof}

Now we are in a position to prove Proposition~\ref{prop:sametwistedhomologies}.

\begin{proof}[Proof of Proposition~\ref{prop:sametwistedhomologies}]
Let $\pi_1$ and $\pi_2$ be good groups. Let  $f\colon \what{\pi_1}\xrightarrow{\cong} \what{\pi_2}$ be an isomorphism. Let $\b\colon \pi_2\to \gl(k,\F_p)$ be a representation. 
Since $\pi_1$ and $\pi_2$   are good we know that the inclusion maps $\pi_j\to \what{\pi_j}$, $j=1,2$ and the map $f$ give us for any $i$ isomorphisms
\[     H^i_{\b^*}(\pi_2;\F_p^k)\xleftarrow{\cong} H^i_{\b^*}(\what{\pi_2};\F_p^k)\xrightarrow{f^*} H^i_{\b^*\circ f}(\what{\pi_1};\F_p^k)\xrightarrow{\cong} H^i_{\b^*\circ f}(\pi_1;\F_p^k).\]
But by Lemma~\ref{lem:homcohom} we also have 
\[     H^i_{\b^*}(\pi_2;\F_p^k)\cong     H_i^{\b}(\pi_2;\F_p^k)\mbox{ and }
  H^i_{\b^*\circ f}(\pi_1;\F_p^k)\cong     H_i^{\b\circ f }(\pi_1;\F_p^k).\]
The proposition follows from combining all these isomorphisms.
\end{proof}

\subsection{Proof of Theorem~\ref{mainthm}}

For the reader's convenience we recall the statement of Theorem~\ref{mainthm}.\\

\noindent \textbf{Theorem~\ref{mainthm}.}
\emph{
Let $N_1$ and $N_2$ be two aspherical 3-manifolds. Suppose $f\colon \what{\pi_1(N_1)}\to \what{\pi_1(N_2)}$ is a regular isomorphism.  Let $\phi\in H^1(N_2;\Z)$.
Then 
\[ (N_2,\phi)\mbox{ is fibered}\quad \Longleftrightarrow \quad (N_1,f^*\phi_1)\mbox{ is fibered}.\]
Furthermore 
\[ x_{N_2}(\phi))=x_{N_1}(f^*\phi).\]}

In the proof of the theorem we will need the following lemma.

\begin{lemma}\label{lem:samedeg}
Let $N_1$ and $N_2$ be two 3-manifolds. Suppose $f\colon \what{\pi_1(N_1)}\to \what{\pi_1(N_2)}$ is a regular isomorphism.  Then for any non-trivial $\phi\in H^1(N_2)$
and any representation $\a\colon \pi_1(N_2)\to \gl(k,\F_p)$ we have 
\[ \deg\big(\Delta_{N_1,\phi\circ f,i}^{\a\circ f}\big)=\deg\big(\Delta_{N_2,\phi,i}^\a\big),\quad  i=0,1,2.\]
\end{lemma}

\begin{proof}
We first point out that for any $n$ we have
$(\a\circ f)\otimes (\phi_n\circ f)=(\a\otimes \phi_n)\circ f$. 
Also, as usual the twisted homology groups in dimensions $0$ and $1$ only depend on the fundamental group. Together with  Theorem~\ref{thm:cav} and Proposition~\ref{prop:sametwistedhomologies} this implies that for any $n$ and $i\in \{0,1\}$ we have
\[  \dim_{\F_p}\left( H_i^{\a\otimes \phi_n}(N_2;\zpktn)\right)= \dim_{\F_p}\left( H_i^{(\a\circ f)\otimes (\phi_n\circ f)}(N_1;\zpktn)\right).\]
For $i=0,1$ the equality of the lemma is now an immediate consequence of
Proposition~\ref{prop:tapdegree}.

Now we consider the case $i=2$. The argument above also shows that 
\[ \deg\left(\Delta_{N_1,\phi\circ f,0}^{\a^*\circ f}\right)=\deg\left(\Delta_{N_2,\phi,0}^{\a^*}\right).\]
The desired equality of degrees now follows from 
Lemma~\ref{lem:delta2} together with Corollary~\ref{cor:profiniteclosedn} and Proposition~\ref{prop:tapdegree}.
\end{proof}

Now we are finally in a position to prove Theorem~\ref{mainthm}.

\begin{proof}[Proof of Theorem~\ref{mainthm}]
Let $N_1$ and $N_2$ be two aspherical 3-manifolds and suppose that we are given a regular isomorphism $f\colon \what{\pi_1(N_1)}\to \what{\pi_1(N_2)}$.
Let $\phi\in H^1(N_2;\R)$. We need to show the following two statements:
\bn
\item  The class $ \phi\in H^1(N_2;\R)$  is fibered if and only if $f^*\phi\in H^1(N_1;\R)$ is fibered.
\item  $x_{N_2}(\phi)=x_{N_1}(f^*\phi)$.
\en

We first prove (1) and (2) in the special case that  $\phi\in H^1(N_2;\R)$
is an integral cohomology class.

\bn
\item 
We suppose that $\phi\in H^1(N_2;\Z)$ is non-fibered. By Theorem~\ref{thm:fv} (1) there exists a representation  $\a\colon \pi_1(N_2)\to \gl(k,\F_p)$ such that  $\deg\left(\Delta_{N_2,\phi,1}^\a\right)=\infty$. 
By Lemma~\ref{lem:samedeg} we have 
$\deg\big(\Delta_{N_1,\phi\circ f}^{\a\circ f}\big)=\infty$.
But by Theorem~\ref{thm:fv} (1) this implies that $f^*\phi\in H^1(N_1;\Z)$ is non-fibered. Running the same argument backwards we see that $\phi$ is fibered if and only if $f^*\phi$ is fibered.
\item 
By Theorem~\ref{thm:fv} (3) there exists a prime $p$ and a representation  $\a\colon \pi_1(N_2)\to \gl(k,\F_p)$ such that the twisted Alexander polynomials $\Delta_{N_2,\phi,i}^\a$, $i=0,1,2$ are non-zero and such that
\[x_{N_2}(\phi)= \frac{1}{k}\left(-\deg \left(\Delta_{N_2,\phi,0}^\a\right)+\deg \left(\Delta_{N_2,\phi,1}^\a\right)-\deg \left(\Delta_{N_2,\phi,2}^\a\right)\right).\]
By Lemma~\ref{lem:samedeg} the degrees on the right hand side are the same for the twisted Alexander polynomials $\Delta_{N_1,\phi\circ f,i}^{\a\circ f}$, $i=0,1,2$. If we combine this observation with the above equality and with Theorem~\ref{thm:fv} (2)  we obtain that  $x_{N_1}(\phi\circ f)\geq x_{N_2}(\phi)$. If we run through this argument with the roles of $N_1$ and $N_2$ switched we see that $x_{N_1}(\phi\circ f)=x_{N_2}(\phi)$. We thus obtained the desired equality.
\en

Summarizing, we just showed that $f^*\colon H^1(N_2;\Z)\to H^1(N_1;\Z)$ is an isometry with respect to the Thurston norms and it defines a bijection of the fibered classes. Since the norms are in particular homogeneous and continues it follows that  $f^*\colon H^1(N_2;\R)\to H^1(N_1;\R)$ is also an isometry with respect to the Thurston norms. Furthermore, Thurston \cite{Th86} showed that the set of fibered classes of a 3-manifold is given by the union on cones on open top-dimensional faces of the Thurston norm ball of the 3-manifolds. The fact that $f^*$ defines an isometry of Thurston norms and that it defines a bijection of integral fibered classes thus also implies that $f^*$ defines a bijection of real fibered classes.
\end{proof}

\subsection{Proof of Theorem~\ref{mainthmknots}}
In the following, recall that 
given a knot $K\subset S^3$ we denote by $X(K):=S^3\sm \nu K$ its exterior and we denote by $\pi(K):=\pi_1(S^3\sm \nu K)$ the corresponding knot group. We say that \emph{$K$ is fibered} if the knot exterior $X(K)$ is a surface bundle over $S^1$. Furthermore, we refer to the minimal genus of a Seifert surface for $K$ as the \emph{genus $g(K)$ of $K$}. 

If $\phi\in H^1(X(K);\Z)$ is a generator, then $K$ is fibered if and only if $\phi$ is a fibered class. Furthermore, it is straightforward to see that $x_{X(K)}(\phi)=\max\{0,2g(K)-1\}$. We thus see that Theorem~\ref{mainthmknots}
is an immediate consequence of the following theorem.

\begin{theorem}\label{mainthmb1=1}
Let $N_1$ and $N_2$ be two 3-manifolds with $H_1(N_1;\Z)\cong H_1(N_2;\Z)\cong \Z$. Suppose there exists an isomorphism $f\colon \what{\pi_1(N_1)}\to \what{\pi_1(N_2)}$.  Let $\phi_i\in H^1(N_i;\Z)$, $i=1,2$ be generators. 
Then $ \phi_1\in H^1(N_1;\Z)$  is fibered if and only if $\phi_2\in H^1(N_2;\Z)$ is fibered. Furthermore
\[ x_{N_1}(\phi_1)=x_{N_2}(\phi_2).\]
\end{theorem}

We point out that Theorem~\ref{mainthmb1=1} is not an immediate consequence of Theorem~\ref{mainthm} since we do not assume that there exists a \emph{regular} isomorphism between the profinite completions of $\pi_1(N_1)$ and $\pi_1(N_2)$.

In the proof of Theorem~\ref{mainthmb1=1} we will need one more lemma.

\begin{lemma}\label{lem:samecyclich1}
Let  $N$ be a $3$-manifold with $H_1(N;\Z)\cong \Z$.  
Let    $\b\colon \pi_1(N)\to \gl(k,\F_p)$ be a representation
and let $\phi_n\colon \pi_1(N)\to \Z_n$ and  $\psi_n\colon \pi_1(N)\to \Z_n$ be two epimorphisms. Then given any $i$ there exists an isomorphism
\[  H_i^{\b\otimes \phi_n}(N;\zpktn)\cong  H_i^{\b\otimes \psi_n}(N;\zpktn).\]
\end{lemma}

\begin{proof}
We denote by $\wti{N}$ the universal cover of $N$.
Since $H_1(N;\Z)\cong \Z$ there exists up to sign only one epimorphism onto $\Z_n$. Therefore it suffices to show that for all $i$ we have
\[  H_i^{\b\otimes \phi_n}(N;\zpktn)\cong  H_i^{\b\otimes -\phi_n}(N;\zpktn).\]
This in turn follows from the observation that 
\[ \ba{rcl} C_*(\wti{N})\otimes_{\Z[\pi_1(N)]}\left(\F_p^k\otimes \Z[\Z_n]\right)_{\b\otimes \phi_n}&\to &  C_*(\wti{N})\otimes_{\Z[\pi_1(N)]}\left(\F_p^k\otimes \Z[\Z_n]\right)_{\b\otimes -\phi_n}\\
\sigma\otimes (v\otimes \sum_{g\in \Z_n} a_gg)&\mapsto&\sigma\otimes (v\otimes \sum_{g\in \Z_n} a_g(-g))\ea \]
is an isomorphism of chain complexes.
\end{proof}

Now we are ready to provide the proof of  Theorem~\ref{mainthmb1=1}.

\begin{proof}[Proof of  Theorem~\ref{mainthmb1=1}]
Let $M$ and $N$ be  3-manifolds with $H_1(M;\Z)\cong H_1(N;\Z)\cong \Z$
and let  $f\colon \what{\pi_1(M)}\to \what{\pi_1(N)}$ be an isomorphism.  Let $\phi\in H^1(M;\Z)$ and $\psi\in H^1(N;\Z)$ be generators. 

It follows from Lemma~\ref{lem:samecyclich1}, Theorem~\ref{thm:cav} and Proposition~\ref{prop:sametwistedhomologies} that for any $n$ and any $i\in \{0,1\}$ we have
\[  \dim_{\F_p}\left( H_i^{\a\otimes \phi_n}(M;\zpktn)\right)= \dim_{\F_p}\left( H_i^{(\a\circ f)\otimes \psi_n}(N;\zpktn)\right).\]
It  follows from  the proof of Lemma~\ref{lem:samedeg} that for any representation $\a\colon \pi_1(N)\to \gl(k,\F_p)$ we have 
\[ \deg\left(\Delta_{M,\phi,i}^{\a\circ f}\right)=\deg\left(\Delta_{N,\psi,i}^\a\right),\quad  i=0,1,2.\]
The argument of the proof of Theorem~\ref{mainthm} now carries over to prove the desired statements.
\end{proof}

\subsection{The profinite completion of the unknot, the trefoil and the Figure-8 knots}

As promised in the introduction we now recall the argument that the profinite completion detects the unknot. 

\begin{lemma} \label{lem:profinitetrivial}
Let $U$ be the unknot. If $K$ is a knot with $\what{\pi(U)}\cong \what{\pi(K)}$, then $K$ is the unknot.
\end{lemma}

\begin{proof}
It is a well-known consequence of Dehn's lemma that a knot $K$ is the unknot if and only if $\pi(K)\cong \Z$. Since $H_1(\pi(K);\Z)\cong \Z$ it follows that a knot is the unknot if and only if $\pi(K)$ is abelian.  By \cite{He87} knot groups are residually finite.
It thus follows that a knot is the unknot if and only if all finite quotients are abelian. Since the finite quotients of a group are the same as the finite quotient of its profinite completion it now follows that a knot $K$ is trivial if and only if all finite quotients of $\what{\pi(K)}$ are abelian. We thus showed that the profinite completion of a knot group determines whether or not the knot is trivial.
\end{proof}

Given a knot $K$ we denote by $X(K)_n$ the $n$-fold cyclic cover of $X(K)$. 
We have the   following elementary lemma.

\begin{lemma}\label{lem:sameh1xn}
Let $J$ and $K$ be two knots such that the profinite completions of $\pi(J)$ and $\pi(K)$ are isomorphic. Then for any $n$ we have $H_1(X(J)_n;\Z)\cong H_1(X(K)_n;\Z)$.
\end{lemma}

\begin{proof}
The isomorphism $f\colon \what{\pi(J)}\to \what{\pi(K)}$ induces an isomorphism
of the profinite completions of $\ker\{\pi(J)\to \Z_n\}$ and $\ker\{\pi(K)\to \Z_n\}$. But if two groups have isomorphic profinite completions the abelianizations have to agree. We thus see that  
\[ H_1(X(J)_n;\Z)\cong H_1(\ker\{\pi(J)\to \Z_n\};\Z)\cong H_1(\ker\{\pi(K)\to \Z_n\};\Z)\cong H_1(X(K)_n;\Z).\]
\end{proof}

This lemma allows us to prove the following proposition.

\begin{proposition}\label{prop:samealex}
Let $J$ and $K$ be two knots such that the profinite completions of $\pi(J)$ and $\pi(K)$ are isomorphic. Then the following hold:
\bn
\item The Alexander polynomial $\Delta_J$ has a zero that is an $n$-th root of unity if and only if $\Delta_K$ has a zero that is an $n$-th root of unity.
\item If neither Alexander polynomial has a zero that is a root of unity, then  $\Delta_J=\pm \Delta_K$.
\en
\end{proposition}

\begin{proof}
Given a set $S$ we henceforth write $|S|=0$ if $S$ is infinite, otherwise we denote by $|S|$ the number of elements.
Let $K$ be a knot and let $n\in \N$.   Fox \cite{Fo56}, see also \cite{We79,Tu86} showed that 
\[ H_1(X(K)_n;\Z)\cong \Z\oplus A\]
where $A$ is a group with 
\[ |A|=\left| \prod_{k=1}^n \Delta_K\left(e^{2\pi ik/n}\right)\right|.\]
 In particular we have $b_1(X(K)_n)=1$ if and only if  if no $n$-th root of unity is a zero of $\Delta_K(t)$.
 
The first statement of the proposition is now an immediate consequence 
of Lemma~\ref{lem:sameh1xn}. The second statement follows from combining the above formula with Lemma~\ref{lem:sameh1xn} and a deep result of Fried~\cite{Fr88}.
\end{proof}

Now we can  also prove the following corollary which we already mentioned in the introduction.\\

\noindent \textbf{Corollary~\ref{cor:distinguish34-intro}.}
\emph{Let $J$ be the trefoil or the Figure-8 knot. If $K$ is a knot with $\what{\pi(J)}\cong \what{\pi(K)}$, then $J$ and $K$ are equivalent.}\\

\begin{proof}
Let $J$ be the trefoil or the Figure-8 knot and let $K$ be another knot with $\what{\pi(J)}\cong \what{\pi(K)}$. It is well-known that $J$ is a fibered knot with $g(J)=1$.  
It follows from Theorem~\ref{mainthmknots} that $K$ is also a fibered knot with $g(K)=g(J)=1$. From \cite[Proposition~5.14]{BZ85} we deduce that $K$ is  either the trefoil  or the Figure-8 knot. 
Thus it suffices to show that the profinite completion can distinguish the trefoil from the Figure-8 knot. But this is a consequence of
Proposition~\ref{prop:samealex} and the fact that the Alexander polynomial of the trefoil is the cyclotomic polynomial $t^{-1}-1+t$ whereas the Alexander polynomial of the Figure-8 knot is $t^{-1}-3t+t=t^{-1}(t-\frac{3+i\sqrt{5}}{2})(t-\frac{3-i\sqrt{5}}{2})$.
\end{proof}

\subsection{Torus knots}\label{section:torusknots}

In this section we prove Theorem \ref{thm:torusknots} stating that each torus knot  is distinguished, among knots, by the profinite completion of its group. 
First we prove that the profinite completion detects torus knots.

\begin{proposition}\label{prop:torusknots}
Let $J$ be a torus knot. If $K$ is a knot with $\what{\pi(J)}\cong \what{\pi(K)}$, then $K$ is a torus knot.
\end{proposition}

Wilton--Zalesskii \cite{WZ14} showed that the profinite completion of the fundamental group determines whether a \emph{closed} 3-manifold is Seifert fibered. 
Proposition~\ref{prop:torusknots} proves the same result for knot complements. Our proof is quite different from the proof provided by Wilton--Zalesskii since our main tool for dealing with hyperbolic JSJ-components is the paper by Long--Reid \cite{LR98}.

\begin{proof}
We argue by contradiction by assuming that either $K$ is a hyperbolic knot or a satellite knot.

In the first case, by \cite{LR98} the group $\pi(K)$ is residually simple. Therefore $\pi(K)$, and so $\pi(J)$, admit co-final towers of finite regular coverings with simple covering groups. But this is impossible for $\pi(J)$, since it has a non-trivial cyclic center. So we can assume that $K$ is a satellite knot.

Since $\pi(J)$ has a non-trivial infinite cyclic center and since abelian subgroups of 3-manifold groups are separable \cite{Ha01}, $\what{\pi(J)}$ has a procyclic center $\what{Z} \cong \what{\mathbb{Z}}$, and the quotient is the profinite completion of the free product of two finite cyclic groups of pairwise prime order, which is centerless by \cite[Thm 3.16]{ZM88}. 

Since $K$ is a satellite knot, the exterior $X(K)=S^3\sm \nu K$ has a non-trivial $JSJ$-splitting. Such a splitting induces a graph-of-groups decomposition of $\pi(K)$. The profinite topology on $\pi(K)$ is efficient for this decomposition which means that the vertex and edge groups are closed and that the profinite topology on $\pi(K)$ induces the full profinite topologies on each vertex and edge group (see \cite[Thm A]{WZ10}). Therefore $\what{\pi(K})$ is a profinite graph of profinite completions of the corresponding vertex and edge groups. Since 
$\what{\pi(K)} \cong \what{\pi(J)}$, the non-trivial procyclic center $\what{Z}$ must belong to each edge group of this graph-of-profinite groups decomposition for
$\what{\pi(K)}$: this follows  from \cite[Thm 3.16]{ZM88}, since the graph of group decomposition of 
$\pi(K)$ is not of dihedral type, otherwise $\pi_(K)$ would be solvable, which is not possible for a non-trivial knot . 

Therefore $\what{Z}$ belongs to the profinite completion of each vertex group of the graph-of-groups decomposition of $\pi(K)$ induced by the JSJ-splitting. Since each vertex corresponds to a hyperbolic or a Seifert piece in the geometric decomposition of $S^3 \setminus K$, it follows from the first step that all the pieces are Seifert fibered and 
thus $X(K)=S^3\sm \nu K$ is a graph 3-manifold.

Let $\what{\mathbb{Z} \times \mathbb{Z}}$ be an edge group of the graph-of-profinite groups decomposition of $\what{\pi(K)}$.  It corresponds to the profinite completion of the corresponding edge group $\mathbb{Z} \times \mathbb{Z}$ of the JSJ-graph-of-groups decomposition of $\pi_1(M)$. The two vertex groups $G_1$ and $G_2$ containing  this edge group 
correspond to the fundamental groups of Seifert pieces of $M$. Each group $G_i, i=1, 2$, is an extension:
$$1 \to \mathbb{Z}= Z_i \to G_i \to \Gamma_i \to 1,$$ where $\Gamma_i, i= 1, 2,$ is a free product of finite cyclic groups. The profinite completion $\what{\Gamma_i}$ is the free product (in the profinite category) of finite cyclic groups (see \cite[Exercice~9.2.7]{RZ10}, and hence is centerless \cite[Thm 3.16]{ZM88}, except perhaps if it is 
$\what{\mathbb{Z}_2 \star \mathbb{Z}_2}$. But in this last case by \cite[Proposition 4.3]{GZ11} $\Gamma_i$ would be isomorphic to $\mathbb{Z}_2 \star \mathbb{Z}_2$ and the Seifert piece would contain an embedded Klein bottle, which is impossible
in $S^3$. Therefore each profinite completion $\what{G_i}, i=1, 2,$ has a procyclic center 
$\what{Z_i}$ which must contain the procyclic center $\what{Z}$ of $\what{\pi(K)}$, since the quotient $\what{G_i}/\what{Z_i}$ is centerless. 

For each prime $p$, the $p$-Sylow subgroup $\what{Z}_{(p)}$ of $\what{Z}$ is of finite index in each $p$-Sylow subgroup $\what{Z_i}_{(p)} \subset \what{Z_i}$, by \cite[Proposition 2.7.1]{RZ10}.
Since the subgroup generated by $Z_1$ and $Z_2$ is of finite index in the edge group 
$\mathbb{Z} \times \mathbb{Z}$, the subgroup generated by $\what{Z_1}$ and $\what{Z_2}$ is of finite index, say $n$, in  the profinite edge group  $\what{\mathbb{Z} \times \mathbb{Z}} \cong  \what{\mathbb{Z}} \times \what{\mathbb{Z}}$. It follows that the $p$-Sylow subgroup $\what{Z}_{(p)} \cong \what{\mathbb{Z}}_{(p)}$ of $\what{Z}$ is of finite index in the p-Sylow subgroup 
$\what{\mathbb{Z}}_{(p)} \times \what{\mathbb{Z}}_{(p)}$ of the edge group for a prime $p$ which does not divide $n$. This gives the desired contradiction, since $\what{\mathbb{Z}}_{(p)} \times \what{\mathbb{Z}}_{(p)}$ is a free $ \what{\mathbb{Z}}_{(p)}$-module of rank $2$.
 \end{proof}

We give now the proof of Theorem \ref{thm:torusknots}.

By Proposition \ref{prop:torusknots} $K$ is a torus knot. If two torus knots $T_{p,q}$ and $T_{r,s}$ have isomorphic profinite completions, their quotients 
$\what{\mathbb{Z}_p \star \mathbb{Z}_q}$ and $\what{\mathbb{Z}_r \star \mathbb{Z}_s}$ by the procyclic center are also isomorphic. Then it follows again  from 
\cite[Proposition 4.3]{GZ11} that the groups $\mathbb{Z}_p \star \mathbb{Z}_q$ and $\mathbb{Z}_r \star \mathbb{Z}_s$ are isomorphic, and so the knots $K$ and $J$ are equivalent.

Another more topological argument is that the two torus knots $T_{p,q}$ and $T_{r,s}$ have the same genus by Theorem \ref{mainthmknots}. So $(p-1)(q-1)=(r-1)(s-1)$.
Moreover the first cyclic covers of the knot complements with maximal Betti number correspond to the $pq$-cover for $T_{p,q}$ and the $rs$-covers for $T_{r,s}$. It follows that 
$pq = rs$ and thus $p + q = r + s$ by the genus equality. Therefore $(p, q) = (r, s)$ or $(p, q) = (s, r)$, in both cases the torus knots $T_{p,q}$ and $T_{r,s}$ are equivalent.

\begin{remark} Torus knots are all commensurable. Moreover they are cyclically commensurable if they have the same genus. However they are distinguished by the profinite completions of their groups. In the next section we discuss the case of hyperbolic knots.
\end{remark}


\subsection{Commensurable knots}\label{section:commensurable}

In this section we prove Theorem \ref{thm:commensurable} which shows that two cyclically commensurable hyperbolic knots can be distinguished by the profinite completions of their groups, provided that their Alexander polynomials are not a product of cyclotomic polynomials.
The proof relies on the following proposition:

\begin{proposition}\label{prop:commensurable} Let $K_1$ and $K_2$ be two distinct and cyclically commensurable hyperbolic knots. Then there is  a compact orientable 3-manifold $Y$ with 
$H_1(Y, \mathbb{Z}) = \mathbb{Z}$ and two coprime integers $p_1$ and $p_2$ such that $X(K_1)$ is a $p_1$-cyclic cover of $Y$ and $X(K_2)$ is a $p_2$-cyclic cover of $Y$. 
\end{proposition}

\begin{proof}

For a hyperbolic knot $K$, the positive solution of the Smith conjecture implies that the subgroup $Z(K)$ of $\operatorname{Isom}^+(X(K))$ which acts freely on $\partial X(K)$ is cyclic, and that the action of 
$Z(K)$ extends to a finite cyclic action on $S^3$ which is conjugated to an orthogonal action. Therefore the orientable orbifold 
$\mathcal{Z}_{K} = X(K)/Z(K)$ is a knot exterior in the quotient $\mathcal{L} = S^3/Z(K)$ which is an orbi-lens space. The notion of orbi-lens space was introduced in \cite[Section 3]{BBCW12}.
An orbi-lens space $\mathcal{L}$ is a  $3$-orbifold whose underlying space $|\mathcal{L}|$ is a lens space and the singular locus $\Sigma(\mathcal{L})$ is a closed submanifold of the union of the cores $C_1 \cup C_2$ of a genus one Heegaard splitting $V_1 \cup V_2$ of the underlying space $|\mathcal{L}|$. In particular, there are coprime positive integers $a_1, a_2 \geq 1$ such that a point of $C_j$ has isotropy group $\mathbb Z / a_j$, and thus the orbifold fundamental group $\pi_1^{orb}(\mathcal{L}) \cong \mathbb Z / (a_1 a_2 |\pi_1(|\mathcal{L}|)|)$. We use the notation $\mathcal{L}(p,q;a)$ to denote such an orbi-lens space
with $a_1 = a$ and $a_2 = 1$. When $a = 1$, $\mathcal{L}(p,q; a)$ is the lens space $L(p,q)$.

By \cite[Propositions 4.7]{BBCW12} if $K_1$ and $K_2$ are two distinct cyclically commensurable hyperbolic knots, then up to orientation preserving homeomorphism, 
$\mathcal{Z}_{K_1} = \mathcal{Z}_{K_2}$. The orbifold $\mathcal{Z} = \mathcal{Z}_{K_1} = \mathcal{Z}_{K_2}$ embeds as a knot exterior in both orbi-lens spaces 
$\mathcal{L}(p_1,q_1; a) = S^3/Z(K_1)$
and  $\mathcal{L}(p_2,q_2; a) = S^3/Z(K_2)$, with $ap_1 = \vert Z(K_1) \vert$, $ap_2 = \vert Z(K_2) \vert$ and $p_1$ coprime to $p_2$, by \cite[Propositions 5.7 and 5.8]{BBCW12}.
In particular the $ap_2$ cyclic cover $M$ of $X(K_1)$ coincides with the $ap_1$ cyclic cover of $X(K_2)$.

By \cite[Thm 1.5]{BBCW12} the orbifold $\mathcal{Z}$ admits a fibration by $2$-orbifolds with base the circle. This fibration lifts to 
fibrations by surfaces over the circle in the exteriors $X(K_1)$, $X(K_2)$  and in their common cyclic covering $M$. So the knots $K_1$ and $K_2$ are fibred knots, and their fibrations lift to the same fibration over the circle in $M$. It follows that the fiber for these three fibrations is  the same surface $F$, and the monodromy  for the fibration on $M$ is $\phi= \phi_1^{ap_2} = \phi_2^{ap_1}\colon F \to F$,
where $\phi_1$ is the monodromy of $K_1$ and $\phi_2$ the monodromy of $K_2$. Since $K_1$ and $K_2$ are hyperbolic knots, the monodromies $\phi_1$ and $\phi_2$ are pseudo-Anosov homeomorphisms of $F$. The uniqueness of the root of a pseudo-Anosov element in the mapping class group of a surface with boundary, \cite[Thm 4.5]{BoPa09}, implies that 
$\phi_1^{p_2} = \phi_2^{p_1}$. Therefore the $p_2$-cyclic cover $N$ of $X(K_1)$ coincides with the $p_1$-cyclic cover of $X(K_2)$. Since $p_1$ and $p_2$ are coprime, 
the deck transformations of the cyclic covers $N \to X(K_1)$ and $N \to X(K_2)$ generate a cyclic subgroup $C$ of $\operatorname{Isom}^+(N)$ of order $p_{1}p_{2}$. The quotient
$Y = N/C$ is cyclically covered by $X(K_1)$ with order $p_1$ and by $X(K_2)$ with order $p_2$. It follows that $Y$ is a manifold, since the order of the isotropy group 
of a point of $Y$ must divide $p_1$ and $p_2$. 

Therefore $Y$ is the exterior of a primitive knot in the lens space $L(p_1,q_1)$,  since its preimage in $S^3$ is the knot $K_1$. 
In particular $H_1(Y, \mathbb{Z}) = \mathbb{Z}$.
\end{proof}

We give now the proof of Theorem \ref{thm:commensurable}.

Let $K_1$ and $K_2$ be two distinct and cyclically commensurable hyperbolic knots. Let $Y$ be the compact orientable 3-manifold provided by Proposition \ref{prop:commensurable}. 
Then $H_1(Y, \mathbb{Z}) = \mathbb{Z}$ and there are two coprime integers $p_1$ and $p_2$ such that $X(K_1)$ is a $p_1$-cyclic cover of $Y$, while $X(K_2)$ is a $p_2$-cyclic cover of $Y$.
Let $\widetilde{Y}$ be the maximal free abelian cover  of $Y$, then  
its first homology group $H_1(\widetilde{Y}, \mathbb{Z})$ is a  torsion module over $\mathbb{Z}[t,t^{-1}]$, whose order is $\Delta_{Y} \in \mathbb{Z}[t,t^{-1}]$. 

Given a non-zero polynomial $p\in \ct$ with top coefficient $c$ and zeros $z_1,\dots,z_m$  we denote by 
\[m(p):=|c|\cdot \prod_{i=1}^{m} \max\{1,|z_i\}\]
 its Mahler measure.
Furthermore, given a topological space $Z$ we write $t(Z):=|\operatorname{Tor}(H_1(Z;\Z))|$. Now let $Z$ be a 3-manifold with $H_1(Z;\Z)\cong \Z$. We denote by $\Delta_Z(t)\in \zt$ the corresponding Alexander polynomial and  given $n\in \N$ we denote by $Z_n$ the cover corresponding to the epimorphism  
$\pi_1(Z) \to H_1(Z, \mathbb{Z}) \cong \mathbb{Z} \to \mathbb{Z}/n\mathbb{Z}$.   Silver--Williams \cite[Theorem~2.1]{SW02}
proved that \[\lim_{ n \to \infty} \frac{\log(t(Z_n))}{n} = \ln(m(\Delta_{Z}(t))).\]

By applying this equation to $X(K_1), X(K_2)$ and $Y$  it follows, see \cite[Section 3.4.1]{Fr13}, that:
$$\ln(m(\Delta_{K_1})) = \lim_{ k \to \infty} \frac{\log(t(X(K_1)_k))}{k} = p_1 \lim_{ k \to \infty} \frac{\log(t(Y_{kp_1}))}{kp_1} = p_{1} \ln(m(\Delta_{Y}(t))),$$ and
$$\ln(m(\Delta_{K_2})) = \lim_{ k \to \infty} \frac{\log(t(X(K_2)_k))}{k} = p_2 \lim_{ k \to \infty} \frac{\log(t(Y_{kp_2}))}{kp_2} = p_{2} \ln(m(\Delta_{Y}(t)))$$.

If $\what{\pi(K_1)}\cong \what{\pi(K_2)}$, then $t(X(K_1)_k) = t(X(K_2)_k)$ for all $k \geq 1$ and thus $p_1 = p_2$ if $\ln(m(\Delta_{Y})) \neq 0$. This is the case if 
the Alexander polynomial $\Delta_{K_1}$ has a zero that is not a root of unity, because then by Kronecker's theorem, see \cite[Theorem~4.5.4]{Pr10}, $\Delta_{K_1}$ has a zero $z$ with $|z|>1$, hence the Mahler measure $m(\Delta_{K_1}) \neq 1$ and $\ln(m(\Delta_{K_1})) \neq 0$. 
But, in this case this contradicts the fact that $p_1$ and $p_2$ are coprime.

\begin{remark} It follows from the proof of Theorem \ref{thm:commensurable} that the Mahler measures of the Alexander polynomials of two distinct and cyclically commensurable hyperbolic knots cannot be equal, except if it is equal to $1$. 

\end{remark}

A knot group is biorderable if it admits a strict total ordering  of its elements which is invariant under multiplication on both sides. By \cite[Thm 1.1]{CR12}, if the group of a fibered knot $K$ is bi-orderable, then the Alexander polynomial of $K$ has a positive real root. Hence, we get the following corollary:

\begin{corollary} Let $K_1$ and $K_2$ be two cyclically commensurable hyperbolic knots. If $\pi(K_1)$ is bi-orderable and $\what{\pi(K_1)}\cong \what{\pi(K_2)}$, then $K_1$ and  $K_2$ are equivalent.
\end{corollary}

\begin{remark} A strong conjecture (see \cite[Conjecture1]{Le14}) states that the volume of the complement of a hyperbolic knot is determined by the growth of homology torsions of its finite covers, and thus by 
the profinite completion of its group. This would imply that cyclically commensurable hyperbolic knots are distinguished by the profinite completion of their group. It would imply furthermore that given a knot $K$ there are only finitely many hyperbolic knots whose groups have a profinite completion isomorphic to $\what{\pi(K)}$.
\end{remark}

\end{document}